\documentclass{interact}
\makeatletter
\def\ps@pprintTitle{%
	\let\@oddhead\@empty
	\let\@evenhead\@empty
	\def\@oddfoot{\centerline{\thepage}}%
	\let\@evenfoot\@oddfoot}
\makeatother

\usepackage{amsmath,amsthm,amsopn,amssymb,mathrsfs,times,newtxtext,newtxmath,upref, systeme, amsmath,xcolor}
\usepackage[mathscr]{eucal}

\usepackage{enumitem}
\setlist{label={$$\roman{enumi}\kern1pt$)$}}

\newtheorem{thm}{Theorem}[section]
\newtheorem{prop}[thm]{Proposition}
\newtheorem{cor}[thm]{Corollary}
\newtheorem*{cor*}{Corollary}
\newtheorem{lema}[thm]{Lemma}
\newtheorem*{lema*}{Lemma}

\numberwithin{equation}{section}
\theoremstyle{definition}
\newtheorem*{Def}{Definition}

\newtheorem{obs}{Remark}

\newcommand{\PI}[2]{\left\langle \,#1 , #2\, \right\rangle}
\newcommand{\K}[2]{\left[\,#1 , #2\, \right]}
\newcommand{\St}{\mathcal{S}}

\newcommand{\HH}{\mathcal{H}}
\newcommand{\KK}{\mathcal{K}}
\newcommand{\M}{\mathcal{M}}

\newcommand{\N}{\mathcal{N}}
\newcommand{\RR}{\mathbb{R}}
\newcommand{\mc}[1]{\mathcal{#1}}
\newcommand{\T}{\mathcal{T}}
\newcommand{\ol}{\overline}
\newcommand{\clran}{\ol{\mathrm{ran}}\,}

\newcommand{\ra}{\rightarrow}
\newcommand{\perpi}{[\perp]}
\DeclareMathOperator{\ran}{ran}
\DeclareMathOperator{\tr}{tr}
\DeclareMathOperator{\esp}{spl}
\DeclareMathOperator{\real}{Re} 
\DeclareMathOperator{\im}{Im}

\begin{document}
\title{Global solutions of approximation problems in Krein spaces}

%


\author{
	\name{M.~Contino\textsuperscript{a,b, c}
		\thanks{CONTACT M.~Contino. Email: mcontino@fi.uba.ar}}
	\affil{
		\textsuperscript {a} Depto. de An\'alisis Matem\'atico, Facultad de Matem\'aticas, Universidad Complutense de Madrid. Plaza de Ciencias 3 (28040), Madrid, Spain. \\ 
		\textsuperscript {b} Instituto Argentino de Matem\'atica ``Alberto Calder\'on'' - CONICET. Saavedra 15, piso 3 (1083), Ciudad Aut\'onoma de Buenos Aires, Argentina.\\  
		\textsuperscript {c} Facultad de Ingenier\'ia - Universidad de Buenos Aires. Paseo Col\'on 850 (1063), Ciudad Aut\'onoma de Buenos Aires, Argentina.
	}
}


%
%
%
\maketitle


\begin{abstract} 
Three approximation problems in Krein spaces are studied, namely the indefinite weighted least squares problem and the related problems of indefinite abstract splines and smoothing. In every case, we analyze 
the existence of a linear and continuous operator that maps each data point to its solution and when the associated operator problem considering the $J$-trace has a solution. 
\end{abstract}


\section{Introduction}
\label{sec:intro}

In \cite{Contino3g}, inspired by \cite{[CorFonMae16]}, with some arguments taken from \cite{Contino}, we studied three well-known approximation and interpolation problems in Hilbert spaces: the weighted least squares problem and its associated abstract splines and smoothing problems. 
We gave conditions under which these minimization problems have a bounded global solution. This means that we not only give conditions that ensured the existence of solutions for every point in a Hilbert space but also guaranteed the existence of an operator that assigns to each point a solution in a linear and continuous way. Additionally, we delved into the operator versions of these problems, considering the $p$-Schatten norm and we established connections between the existence of solutions for these problems and the existence of bounded global solutions.

In this work we are interested in extending these results for approximation and interpolation problems to the Krein spaces setting. We consider the indefinite weighted least squares problem and the indefinite abstract splines and smoothing problems. 

The indefinite weighted least squares problems in Krein spaces were studied by Hassibi et al.  \cite{HassibipartI} in their ``pointwise'' form. These problem were also studied  in  \cite{GiribetKrein} for linear operators on infinite-dimensional spaces and in \cite{Hassibietal,HassibipartII}, for matrices with complex entries.

Let $(\HH, \K{\cdot}{\cdot})$ and $(\KK, \K{\cdot}{\cdot})$ be separable complex Krein spaces, denote by $L(\KK, \HH)$ the set of bounded linear operators from $\KK$ to $\HH,$ and set $L(\HH):=L(\HH,\HH).$  Given $A\in L(\KK, \HH)$, $W\in L(\HH)$ selfadjoint and $x \in \HH,$ an \emph{indefinite weighted least squares solution} of the equation $Az = x$  is a vector $u\in \KK$ such that
\begin{equation}
\label{WLSPI}
\K{W(Au-x)}{Au-x} \leq \K{W(Az-x)}{Az-x} \textrm{ for every }z \in \KK.
\end{equation}


Some aspects of the indefinite metric spaces theory have been applied to adaptive filter theory \cite{Hassibietal2, HassibipartII}. Also, the introduction of Krein spaces in $\HH^{\infty}$ estimation and control techniques has enabled the adaptation of traditional tools of control theory to $\HH^{\infty}$ control problems \cite{Hassibietal3}. Also, a problem related to weighted indefinite least squares problems is the \emph{linear state estimation in $\HH^{\infty}$ spaces}, see  \cite[Section 1]{GiribetKrein}.

In this work, we focus on establishing conditions for the existence of global solutions of \eqref{WLSPI}. 
For a positive (semidefinite) weight $W,$ the notion of $W$-inverse was introduced by Mitra and Rao in the case of matrices \cite{Mitra}, and later on extended to Hilbert space operators in \cite{[CorFonMae13], Contino}. An operator $G \in L(\HH, \KK)$ is an indefinite $W$-inverse if for each $x \in \HH$, $Gx$ is an indefinite weighted least squares solution of $Az=x.$ 
That is, $G$ is a bounded global solution of \eqref{WLSPI}.  In this work, necessary and sufficient conditions for the existence of such a solution are given.

In the Hilbert space setting an associated minimizing problem can be considered in the context of unitarily invariant norms, in particular, the $p$-Schatten class norms 
$\| \cdot \|_p,$ for $1 \leq p <\infty,$ \cite{Ringrose}. If $W \in L(\HH)$ is positive (semidefinite) such that $W^{1/2}$ is in the $p$-Schatten class $S_p$ for some $1 \leq p <\infty,$ the problem of studying the existence of $W$-inverses is related to the operator Procrustes problem: $$\underset{X \in L(\HH)}{\min} \ \|W^{1/2}(BX-C)\|_p,$$ see \cite{Contino}.

Inspired by the work of Kintzel on an indefinite Procrustes problem expressed as a max-min problem on traces of matrices \cite{Ulric}, in \cite{Contino33}, we defined a $J$-trace, $\tr_J$, for some signature operator $J$ of $\HH.$ 
Here, we study the corresponding  \textit{indefinite operator weighted least squares problem}: given $A\in L(\KK, \HH)$ (not necessarily with closed range) and $W\in L(\HH)$ selfadjoint such that  $W\in S_1,$ analyze the existence of
\begin{equation}
\label{eqpI}
	\min_{X\in L(\HH, \KK)} \tr_J ((AX-I)^{\#}W(AX-I)),
\end{equation}
Problem \eqref{eqpI} was studied in \cite{Contino33} when $A \in L(\HH)$ is a  closed range operator. 
In a similar fashion we study the indefinite spline and smoothing problems. A Hilbert space formulation of abstract splines and smoothing problems was introduced by Atteia \cite{[Att65]} and extended by several authors: see for example Anselone and Laurent \cite{[AnsLau68]} 
and the surveys by Champion, Lenard and Mills \cite{[ChaLenMil96],[ChaLenMil00]}. The abstract spline and smoothing problems were generalized to bounded linear operators in \cite{Spline} and have been applied in many areas, such as approximation theory, numerical analysis and statistics, among others \cite{Holladay}, 
spline functions in indefinite metric spaces have already been studied in \cite{Canu,Loosli} to solve numerical aspects related to learning theory problems. 

Consider  $T, V \in L(\HH).$ The \emph{indefinite spline problem} can be stated as follows: given  $f_0 \in \ran V$ (the range of $V$), analyze if there exists
$$\min\K{Tx}{Tx}, \mbox{ subject to } Vx=f_0. $$  
The \textit{indefinite smoothing problem}: given $h_0\in \HH$ and $\rho \in \RR \setminus\{0\},$ study whether there exists
$$\underset{x\in \HH}{\min}(\K{Tx}{Tx}+\rho \K{Vx-h_0}{Vx-h_0}).
$$

The advantage of the above regularized problem is that it can be restated as an indefinite least-squares problem. These problems have been thoroughly studied before, both in finite-dimensional spaces \cite{Hassibietal,HassibipartI,HassibipartII} and in infinite dimensional Krein spaces \cite{Bognar,GiribetKrein,GiribetKreinII}.
The indefinite abstract smoothing problem in Krein spaces was initially studied in \cite{GiribetKreinsplines} but for a linear constraint. Also, in \cite{CanuII} another version of this abstract smoothing problem was studied. In \cite{UMA,JOTA}, non-convex constraints were considered.


For each of the three indefinite problems presented, we analyze if the problem has a solution for every point of the Krein space, we study the existence of a linear and continuous operator that maps each data point to its solution, and when the associated operator problem considering the $J$-trace has a solution. 


The paper is organized as follows. Section 2 fixes notation and recalls the basics of Krein spaces and the notion of (weak) complementability between a selfadjoint operator and a closed subspace of $\HH$. Also, certain properties of the Schur complement in the Krein space setting are collected.

In Section 3,  global solutions of the indefinite weighted least squares problem \eqref{WLSPI}, where the range of $A$ is not necessarily closed is studied. It is proved that  \eqref{WLSPI} admits a solution for every $x \in \HH$ if and only if $A$ admits a $W$-inverse, or equivalently, if \eqref{eqpI} admits a solution.

In Section 4, we give conditions for the existence of bounded global solutions in both the indefinite spline and the smoothing problems. We also compare the solutions of these problems to the solutions of the corresponding operator problems in the $J$-trace. Finally, indefinite $W$-optimal inverses are defined and they turn out to be the corresponding global solutions of the indefinite smoothing problem.

\section{Preliminaries}
\label{sec:preliminaries}

We assume that all Hilbert spaces are complex and separable. If $\HH, \KK$ are separable complex Hilbert spaces,  $L(\KK, \HH)$  is the set of bounded linear operators from $\KK$ to $\HH,$ for short we write $L(\HH):=L(\HH,\HH)$ and $L(\HH)^+$ denotes the cone of positive semidefinite operators in $L(\HH).$  

The range and nullspace of any $A \in L(\HH)$ are denoted by $\ran A$ and $\ker A$, respectively. Given two operators $S, T \in L(\HH),$ the notation  $T \leq_{\HH} S$ signifies that $S-T \in L(\HH)^+.$ 
For any $T \in L(\HH),$ $\vert T \vert := (T^*T)^{1/2}$ is the modulus of $T$ and $T=U\vert T\vert$ is the polar decomposition of  $T,$ with $U$ the partial isometry such that $\ker U=\ker T.$
The symbol $A^{\dagger}$ denotes the Moore-Penrose inverse of the operator $A \in L(\HH).$

The direct sum of two closed subspaces $\M$ and $\N$ of $\HH$ is represented by $\M \dot{+} \N,$ and $\M \oplus \N$ if $\M \subseteq \N^{\perp}$. 
If $\HH$ is decomposed as $\HH=\M \dot{+} \N,$ the projection onto $\M$ with nullspace $\N$ is denoted by $P_{\M {\mathbin{\!/\mkern-3mu/\!}} \N}$ and abbreviated $P_{\M}$ when $\N = \M^{\perp}.$ 

We recall that, for a Banach space $(\mc{E}, \Vert \cdot \Vert)$ and an open set $\mathcal U \subseteq \mc{E},$ a function $F: \mc{E} \rightarrow \mathbb{R}$ is said to be \emph{Fr\'echet differentiable} at $X_0 \in \mathcal U$ if there exists $DF(X_0)$ a bounded linear functional such that
$$\lim\limits_{Y\rightarrow 0} \frac{|F(X_0+Y)-F(X_0) - DF(X_0)(Y)|}{\Vert Y \Vert}=0.$$
If $F$ is Fr\'echet differentiable at every $X_0 \in \mc{E}$, $F$ is called Fr\'echet differentiable  on $\mc{E}$  and the function $DF$ which assigns to every point $X_0 \in \mc{E}$ the derivative $DF(X_0),$ is called the Fr\'echet derivative of the function $F.$ If, in addition, the derivative $DF$ is continuous, $F$ is  said to be a \emph{class $\mc{C}^1$-function}, in symbols, $F \in \mc{C}^1(\mc{E}, \mathbb{R}).$
\subsection*{\textbf{Krein Spaces}}



An indefinite metric space $(\HH, \K{ \ }{ \ })$ is a {\emph{Krein space}} if it admits a decomposition as an $\K{ \ }{ \ }$-orthogonal direct sum 
\begin{equation} \label{fundamentaldecom}
	\HH=\HH_+ \ [\dotplus] \ \HH_-,
\end{equation} 
where $(\HH_+, \K{ \ }{ \ })$ and $(\HH_-, -\K{ \ }{ \ })$ are Hilbert spaces. Any decomposition with these properties is called a {\emph{fundamental decomposition}} of  $\HH.$


Every fundamental decomposition of $\HH$ has an associated {\emph{signature operator}}: $J:=P_{+} - P_{-}$ with $P_{\pm}:=P_{\HH_{\pm}  {\mathbin{\!/\mkern-3mu/\!}} \HH_{\mp}}.$ The indefinite metric and the inner product corresponding to a fundamental decomposition of $\HH$ with signature operator $J$ are related to each other by 
$$\PI{x}{y}_J=\K{Jx}{y} \quad x, y \in \HH.$$

The {\emph{orthogonal companion}} of a set $\St$ in  $\HH$, which we denote by $\St^{\perpi}$, is the subspace of those  $h \in \HH$ such that $[h,x] = 0$ for all $x\in \St$.
By $\St^{\bot_J}$ we denote the orthogonal  complement of $\St$ with respect to $\PI{\cdot}{\cdot}_J,$ and $\oplus_J$ is defined likewise. Note that, for any signature operator $J$ of $\HH,$
$$\St^{\perpi}=J\St^{\perp_J}=(J\St)^{\perp_J}.$$

If $\HH$ and $\KK$ are Krein spaces, $L(\KK, \HH)$ stands for the vector space of  all the linear operators from $\KK$ to $\HH$ which are bounded in an associated Hilbert space $(\HH, \PI{ \ }{ \ }_J).$Since the norms generated by different fundamental decompositions of a Krein space $\HH$ are equivalent (see, for instance, \cite[Theorem 7.19]{Azizov}), $L(\KK, \HH)$ does not depend on the chosen underlying Hilbert space.

The symbol $T^{\#}$ stands for the $\K{ \ }{ \ }$-adjoint of $T \in L(\HH)$. It can be checked that, for any signature operator $J$ of $\HH,$ it holds that $$T^{\#}=JT^{*_J}J,$$ where $T^{*_J}$ denotes the $\PI{ \ }{ \ }_J$-adjoint of $T.$ 
The set of the operators $T \in L(\HH)$ such that $T=T^{\#}$
is denoted by $L(\HH)^s$. If $T\in L(\HH)^s$ and $\K{Tx}{x} \geq 0 \mbox{ for every } x \in \HH,$ $T$ is said to be {\emph{positive (semidefinite)}}; the notation  $S \leq T$ signifies that $T-S$ is positive.

Let $J$ be a signature operator of $\HH.$  Given $T \in L(\HH)$ and $\St $ a closed subspace of $\HH,$ the symbols $T^{\dagger_J}$ and $P_{\St}^J$ denote the Moore-Penrose of $T$ and the orthogonal projection onto $\St,$ respectively in the Hilbert space $(\HH, \PI{ \ }{ \ }_J).$

Given $W \in L(\HH)^s$ and $\St$ a closed subspace of $\HH,$ we say that $\St$ is $W$-\emph{positive} if 
$\K{Ws}{s} > 0$ for every $s \in \St, \ s\not =0.$ $W$-\emph{nonnegative}, $W$-\emph{neutral}, $W$-\emph{negative} and $W$-\emph{nonpositive} subspaces are defined likewise. 

Standard references on Krein space theory are \cite{Azizov}, \cite{Bognar} and  \cite{DR1}.

\subsection*{Schur complement in Krein Spaces}

%
The notion of Schur complement (or shorted operator) of $A$ to $\St$ for a positive operator $A$ on a Hilbert space $\HH$ and $\St \subseteq \HH$ a closed subspace, was introduced by M.G.~Krein \cite{Krein}. He proved that the set $\{ X \in L(\HH): \ 0\leq_{\HH} X\leq_{\HH} A \mbox{ and } \ran X \subseteq \St^{\perp}\}$ has a maximum element, which he defined as the {{Schur complement}} $A_{/ \St}$ of $A$ to $\St.$ This notion was later rediscovered by Anderson and Trapp \cite{Shorted2}. 

%
In \cite{AntCorSto06} Antezana et al., extended the notion of Schur complement to any bounded operator $A$ satisfying a weak complementability condition with respect to a given pair of closed subspaces $\St$ and $\T.$ In particular, if $A$ is a bounded selfadjoint operator, $\St=\T$ and the matrix decomposition of $A$ with respect to $\St$ is, $A=\begin{pmatrix} a & b\\ b^* & c \end{pmatrix},$ the condition of \emph{$\St$-weak complementability} reads $\ran b \subseteq \ran \vert a \vert^{1/2},$ which is automatic for positive operators. 


In \cite{Contino4}, the notions of $\St$-complementability, $\St$-weak complementability and the Schur complement were extended to the Krein space setting. 

In what follows $\HH$ is a Krein space. 

\begin{Def}  Let $W \in L(\HH)^s$ and $\St$ be a closed subspace of $\HH.$ The operator $W$ is called $\St$-\emph{complementable} if 
	$$\HH=\St + (W\St)^{\perpi}.$$
\end{Def}


In a similar way the $\St$-weak complementability in Krein spaces, with respect to a fixed signature operator $J,$ is defined.

\begin{Def} Let $W \in L(\HH)^s$ and $\St$ be a closed subspace of $\HH.$ The operator $W$ is $\St$-\emph{weakly complementable} with respect to a signature operator $J$ if $JW$ is $\St$-weakly complementable in $(\HH, \PI{ \ }{ \ }_J).$ 
\end{Def}


The $\St$-weak complementability of $W$ does not depend on the signature operator, see \cite[Theorem 4.4]{Contino4}.
Then, we simply say that $W$ is $\St$-weakly complementable, whenever $W$ is $\St$-weakly complementable with respect to a signature operator $J.$

%

\begin{Def} Let $W \in L(\HH)^s,$ $\St$ be a closed subspace of $\HH$ and $J$ a signature operator. Suppose that $W$ is $\St$-weakly complementable. The \emph{Schur complement} of $W$ to $\St$ corresponding to $J$ is
	$$W_{/ [\St]}^J =J (JW)_{ / \St}.$$
\end{Def}

In \cite[Theorem 4.5]{Contino4} it was proved that the Schur complement does not depend on the fundamental decomposition of $\HH.$ Henceforth we write $W_{/ [\St]}$ for this operator.

\section{Global solutions of weighted least squares problems}\label{section3}

Given $A\in L(\KK, \mathcal H)$, $W\in L(\HH)^{s}$ and $x \in \HH$, an \emph{indefinite weighted least squares solution} (or $W$-ILSS) of the equation $Az = x$  is a vector $u\in \KK$ such that
\begin{equation}
\label{WLSP}
\K{W(Au-x)}{Au-x} \leq \K{W(Az-x)}{Az-x} \textrm{ for every }z \in \KK.
\end{equation}
It can be proved  that  $u \in \KK$ is an indefinite weighted least squares solution of $Az =x$ if and only if $\ran A$ is $W$-nonnegative and $Au - x \in (W\ran A)^{\perpi}=\ker A^{\#}W,$ see \cite[Theorem 8.4]{Bognar} and \cite[Lemma 3.1]{GiribetKrein}. 
\begin{prop} \label{propWLSS} Let $A\in L(\KK,\mathcal H)$, $W\in L(\HH)^{s}$ and $x \in \HH.$ Then $u \in \KK$ is a $W$-ILSS of $Az =x$ if and only if $\ran A$ is $W$-nonnegative and $A^{\#}W(Au-x)=0.$
\end{prop}

\smallskip

To study the existence of solutions  of problem \eqref{WLSP} for every $x\in \HH$, the concept of \emph{indefinite weighted inverse} is a fundamental tool. We extend the concept of weighted inverses first defined by Mitra \cite{Mitra} and later generalized to Hilbert space operators in \cite{[CorFonMae13], Contino}. See also \cite{Contino3}.

\begin{Def}\rm Given $A\in L(\KK,\mathcal H)$ and $W\in L(\HH)^{s}.$ An operator $G \in L(\HH,\KK)$ is called an \emph{indefinite $W$-inverse} of $A$ (or an \emph{indefinite bounded global solution} of problem \eqref{WLSP}) if for each $x \in \HH$, $G x$ is a $W$-ILSS of $Az=x$, i.e.,
	$$
	\K{W(AGx-x)}{AGx-x} \leq \K{W(Az-x)}{Az-x} \textrm{ for every }z \in \KK.
	$$
\end{Def}

Note that $A$ has an indefinite $W$-inverse if and only if the problem \eqref{WLSP} admits a solution for every $x\in \mathcal H.$ Also, it is possible to assign a weighted indefinite least squares solution of $Az=x$ to each $x \in \HH$ in a linear and continuous way.
\medskip

In the following we give necessary and sufficient conditions for problem \eqref{WLSP} to admit a bounded global solution, when $A$ is not  necessarily a closed range operator (c.f. \cite[Theorem 4.4]{Contino33}).

\begin{thm}\label{sol global sii  A-inversa} Let $A\in L(\KK, \mathcal H)$ and $W\in L(\HH)^{s}.$ 
	Then the following statements are equivalent:
	\begin{enumerate}
		\item[i)] $Az=x$ admits a $W$-ILSS for every $x \in \HH;$ 
		\item[ii)]  $\ran A +[W(\ran A)]^{\perpi}=\HH$ and $\ran A$ is $W$-nonnegative;
		\item[iii)] the normal equation \begin{equation} \label{ecuacion normal}
			A^{\#}W(AX-I)=0
		\end{equation} admits a solution and $\ran A$ is $W$-nonnegative;
		\item[iv)] the operator $A$ admits an indefinite $W$-inverse.
	\end{enumerate}
	In this case, the set of indefinite $W$-inverses of $A$ is the set of solutions of \eqref{ecuacion normal}.
\end{thm}

\begin{proof} Since $W\in L(\HH)^{s},$ it holds that $[W(\ran A)]^{\perpi}=W^{-1}(\ran A^{\perpi})=\ker A^{\#}W.$
	
	$i) \Rightarrow ii)$: By Proposition \ref{propWLSS}, Problem \eqref{WLSP} admits a solution for every $x\in \HH$ if and only if  $\ran A$ is $W$-nonnegative and there exists $u \in \KK$ such that $A^{\#}WAu=A^{\#}Wx$ for every $x \in \HH$ or, equivalently $\ran A^{\#}W=\ran A^{\#}WA.$ Then
	$$\HH=(A^{\#}W)^{-1}(\ran A^{\#}W)=(A^{\#}W)^{-1}(A^{\#}W(\ran A))=\ran A+[W(\ran A)]^{\perpi}.$$
	
	$ii) \Rightarrow iii)$: Suppose that $\ran A +[W(\ran A)]^{\perpi}=\HH.$ Then 
	$\ran A^{\#}W=\ran A^{\#}WA$ and $iii)$ follows by applying Douglas' Lemma \cite{Douglas}.
	
	$iii) \Rightarrow iv)$: There exists $X_0\in L(\HH,\KK)$ such that $A^{\#}WAX_0=A^{\#}W$ if and only if  $A^{\#}W(AX_0x-x)=0$ for every $x\in \HH.$ Since $\ran A$ is $W$-nonnegative, by Proposition \ref{propWLSS}, $X_0$ is a $W$-inverse of $A$.
	
	$iv) \Rightarrow i)$: It is straightforward.
	
	In this case, we have also proved that the set of $W$-inverses of $A$ is the set of solutions of \eqref{ecuacion normal}.
\end{proof}

\smallskip
\begin{cor} \label{corwc} Let $A\in L(\KK, \mathcal H)$ and $W\in L(\HH)^{s}.$   If $A$ admits an indefinite $W$-inverse then $W$ is $\clran A$ complementable.
\end{cor}
\begin{proof}
Suppose that $A$ admits a $W$-inverse. By Theorem \ref{sol global sii  A-inversa}, $$\HH=\ran A +[W(\ran A)]^{\perpi} \subseteq \clran A + [W(\clran A)]^{\perpi},$$ because  $[W(\ran A)]^{\perpi}=[W(\clran A)]^{\perpi}.$
\end{proof} 

\begin{obs}  Let $A\in L(\KK, \mathcal H)$ and suppose that $W=I.$ If $A$ admits an indefinite inverse then $\clran A$ is a \emph{regular} subspace of $\HH,$ that is, $\HH=\clran [\dotplus] \ran A^{\perpi}.$ See \cite[Section 1.7]{Azizov}.
\end{obs}

\subsection*{Indefinite operator weighted least squares problems}
In this subsection we study weighted least squares problems for operators considering the $J$-trace.  We denote by $S_p$ the $p${\emph{-Schatten class}} for $1 \leq p < \infty.$ The reader is referred to \cite{Ringrose} for further details on $S_p$-operators in Hilbert spaces. 

Let $(\HH, \K{ \ }{ \ })$ be a Krein space. If $J$ is a signature operator for $\HH,$ fix the Hilbert space $(\HH, \PI{ \ }{ \ }_J),$ where $\PI{x }{y}_J=\K{Jx}{y}$ for all $x, y \in \HH.$ 
The operator $T$ belongs to the Schatten class $S_p(J)$ if $T \in S_p$ when viewed as acting on the associated Hilbert space $(\HH, \PI{ \ }{ \ }_J).$

By \cite[Lemma 5.3]{Contino33}, if $T \in S_p(J_a)$ for some fundamental decomposition of $\HH$ with signature operator $J_a$ then $T \in S_p(J_b)$ for any other fundamental decomposition of $\HH$ with signature operator $J_b.$  So we just write $S_p$ instead of $S_p(J).$

\begin{Def} Let $(\HH, \K{ \ }{ \ })$ be a separable Krein space with signature operator $J$ and fix the associated Hilbert space $(\HH, \PI{ \ }{ \ }_J).$ If $T \in S_1$ and $\{e_n : n\in \mathbb{N} \}$ is an orthonormal basis of $(\HH, \PI{ \ }{ \ }_J),$ then the $J${\emph{-trace}} of $T,$ denoted by $\tr_{J}(T),$ is defined as 
$$\tr_{J}(T):=\sum_{n=1}^{\infty} \K{Te_n}{e_n}.$$
\end{Def}

Notice that $$\tr_J(T)=\sum_{n=1}^{\infty} \K{Te_n}{e_n}=\sum_{n=1}^{\infty} \PI{JTe_n}{e_n}_J=\tr(JT),$$ where $\tr(T)$ denotes the \emph{trace} of the operator $T$ in the Hilbert space $(\HH,\PI{ \ }{ \ }_J),$ see \cite{Ringrose}. 
Whence the $J$-trace of $T$ does not depend on the particular choice of the orthonormal basis  \cite[Lemma 2.2.1]{Ringrose}. See \cite[Lemma 5.4]{Contino33}, for the basic properties of the $J$-trace. 

As noted in \cite[Section 5]{Contino33}, if
$\K{S}{T}_J:=\tr_{J}(T^{\#}S)$ for $S,T \in S_2$ then $(S_2, \K{ \ } { \ }_J)$ is a Krein space. Moreover,
$\tr_{J}(T^{\#}T)=\Vert P_+ T \Vert_2^2-\Vert P_- T \Vert_2^2,$ where $P_{\pm}=\frac{I\pm J}{2}$ and $\Vert T \Vert_2$ denotes the $2$-Schatten  norm of the operator $T \in S_2.$

\medskip Let $(\HH, \K{ \ }{ \ })$ be a separable Krein space with signature operator $J$ and fix the associated Hilbert space $(\HH, \PI{ \ }{ \ }_J).$ Given $A\in L(\KK, \mathcal H)$ and $W\in S_1 \cap L(\HH)^{s},$ 
we analyze if there exists
\begin{equation}
\label{eqp}
\min_{X\in L(\HH,\KK)} \tr_{J}((AX-I)^{\#}W(AX-I)).
\end{equation}
We will refer to problem \eqref{eqp} as the \emph{indefinite operator weighted least squares problem}. In order to study \eqref{eqp}, we introduce the following associated problem: given $A\in L(\KK, \mathcal H)$ and $W\in L(\HH)^{s},$ analyze the existence of 
\begin{equation} \label{eqop}
	\min_{X\in L(\HH,\KK)} (AX-I)^{\#}W(AX-I),
\end{equation}
in the order induced in $L(\HH)$ by the cone of positive operators.
By studying problems \eqref{eqp} and \eqref{eqop} we will relate the existence of solutions of \eqref{eqp} to the existence of bounded global solutions of \eqref{WLSP}.

In \cite{Contino33}, problems  \eqref{eqp} and \eqref{eqop} were studied for $A\in L(\mathcal H)$  such that $\ran A$ is closed. The results obtained in \cite{Contino33} are also valid in the general case.

\begin{prop} \label{PropA} Let  $A\in L(\KK, \mathcal H)$ and $W\in S_1 \cap L(\HH)^{s}.$ 
Then the following statements are equivalent:
	\begin{enumerate}
		\item [i)]  $\ran A$ is $W$-nonnegative and  there exists $$\min_{X\in L(\HH,\KK)} \tr_{J}((AX-I)^{\#}W(AX-I))$$ for any signature operator $J;$
		\item [ii)] $\ran A$ is $W$-nonnegative and $\ran A+[W(\ran A)]^{\perpi}=\HH;$
		\item [iii)] there exists $\min_{X\in L(\HH,\KK)} (AX-I)^{\#}W(AX-I).$
	\end{enumerate}
In this case, $$\min_{X\in L(\HH,\KK)} \tr_{J}((AX-I)^{\#}W(AX-I))=\tr_{J}(W_{ /[\clran A]})$$ and $$\min_{X\in L(\HH,\KK)} (AX-I)^{\#}W(AX-I)=W_{ /[\clran A]}.$$
\end{prop}

\begin{proof}  $i) \Rightarrow ii)$: Suppose that $i)$ holds. Then, by similar arguments as those found in \cite[Theorem 5.8]{Contino33}, it follows that the normal equation $A^{\#}W(AX-I)=0$ admits a solution. So that, by Theorem \ref{sol global sii  A-inversa}, $ii)$ holds. 
	
$ii) \Rightarrow iii)$: If $ii)$ holds, by Theorem \ref{sol global sii  A-inversa}, $A$ admits an indefinite $W$-inverse, say $X_0 \in L(\HH, \KK).$ Then, for each $x \in \HH,$ 
$$\K{W(AX_0x-x)}{AX_0x-x} \leq \K{W(Az-x)}{Az-x} \textrm{ for every }z \in \KK.$$ In particular, given $X \in L(\HH, \KK),$ consider $z=Xx.$ So that, for every $x \in \HH,$
$$\K{W(AX_0-I)x}{(AX_0-I)x} \leq \K{W(AX-I)x}{(AX-I)x}.$$
Therefore, $ (AX_0-I)^{\#}W(AX_0-I)\leq (AX-I)^{\#}W(AX-I),$ for every $X \in L(\HH, \KK),$ and $iii)$ holds.

$iii)  \Rightarrow i)$: Suppose that there exists $\min_{X\in L(\HH,\KK)} (AX-I)^{\#}W(AX-I).$ Then, by similar arguments as those in the proof of \cite[Theorem 5.8]{Contino33}, there exists $\min_{X\in L(\HH,\KK)} \tr_{J}((AX-I)^{\#}W(AX-I))$ for any signature operator $J.$ Also, as in the proof of \cite[Theorem 4.4]{Contino33}, $X_0x$ is a $W$-ILSS of $Az=x$ for every $x \in \HH.$ Then,  by Theorem \ref{sol global sii  A-inversa}, $\ran A$ is $W$-nonnegative and $i)$ follows.

	
In this case, let $X_0 \in L(\HH,\KK)$ be a solution of problem \eqref{eqop}, i.e., $$(AX_0-I)^{\#}W(AX_0-I)=\underset{X\in L(\HH)}{\min} (AX-I)^{\#}W(AX-I).$$

In particular, $(AX_0-I)^{\#}W(AX_0-I) \in L(\HH)^{s}, $ $(AX_0-I)^{\#}W(AX_0-I) \leq W$ and, by Theorem \ref{sol global sii  A-inversa}, $A^{\#}W(AX_0-I)=0.$ 
But, since $$A^{\#}[(AX_0-I)^{\#}W(AX_0-I)]=A^{\#}X_0^{\#}A^{\#}W(AX_0-I)-A^{\#}W(AX_0-I)=0,$$ we have that $\ran((AX_0-I)^{\#}W(AX_0-I)))\subseteq \ran A^{\perpi}.$ Let $Z \in L(\HH)^{s}$ be such that $Z \leq W$ and $\ran Z \subseteq \ran A^{\perpi}.$ Then, since $A^{\#}Z=ZA=0,$ $$Z=(AX_0-I)^{\#}Z(AX_0-I) \leq (AX_0-I)^{\#}W(AX_0-I).$$ Therefore  
\begin{align*}
\underset{X\in L(\HH)}{\min} (AX-I)^{\#}W(AX-I)&=(AX_0-I)^{\#}W(AX_0-I)\\
&=\max \  \{ Z \in L(\HH)^{s}: \ Z \leq W,  \ran Z \subseteq \ran A^{\perpi}\}\\
&= W_{ /[\clran A]},
\end{align*}
because by Corollary \ref{corwc}, $\clran A$ is $W$-complementable, $\clran A$ is $W$-nonnegative  and we used \cite[Proposition 4.8]{Contino4}. Finally, by \cite[Theorem 5.8]{Contino33}, 
$$\min_{X\in L(\HH)} \tr_{J}((AX-I)^{\#}W(AX-I)=\tr_{J}((AX_0-I)^{\#}W(AX_0-I)=\tr_{J}(W_{ /[\clran A]}),$$ for any $X_0$ solution of problem \eqref{eqop}. 
\end{proof}

\medskip
The results of the section are collected in the next corollary.

\begin{cor} \label{CorA} Let  $A\in L(\KK, \mathcal H)$ and $W\in S_1 \cap L(\HH)^{s}.$  Then the following statements are equivalent:
	\begin{enumerate}
		\item [i)] $Az=x$ admits an $W$-ILSS for every $x \in \HH;$ 
		\item [ii)] $A$ admits an indefinite $W$-inverse, i.e., for every $x \in \HH,$ $Az=x$ admits a $W$-ILSS, $Gx,$ with $G \in L(\HH,\KK);$
		\item [iii)] $\ran A$ is $W$-nonnegative and  there exists $$\min_{X\in L(\HH,\KK)} \tr_{J}((AX-I)^{\#}W(AX-I))$$ for any signature operator $J.$
	\end{enumerate}
\end{cor}

\section{Global solutions of spline and  smoothing problems}
In this section, we analyze when the indefinite spline and smoothing problems admit global solutions. As in the section above, we relate the corresponding bounded global solutions to the solutions of the associated operator minimization problems considering the $J$-trace.

\subsection*{Indefinite splines problems}
Given a Krein space $\HH$ consider  $T \in L(\HH),$ $V \in L(\HH)$ and $f_0 \in \ran V$, we study the existence of
\begin{equation} \label{spline11}
\min \K{Tx}{Tx}, \mbox{ subject to } Vx=f_0. 
\end{equation}

Suppose that $Vh_0=f_0.$ Then problem \eqref{spline11} is equivalent to study when the set
\begin{equation} 
\label{spline}
\begin{split} 
\esp(T,V,h_0) \!:= \!\{ h \in h_0+ \ker V \!: \! \K{Th}{Th}=\underset{z \in \ker V}{\min} \K{T(h_0+z)}{T(h_0+z)}\} \!
\end{split}
\end{equation}
is not empty. Since $T$ and $V$ are fixed along this work,  $\esp(T, V,h_0)$ is shortened to $\esp(h_0).$
We will refer to problem \eqref{spline} as the \emph{indefinite abstract spline problem}.

\begin{obs}\label{remarksp}
Notice that $x_0$ is  solution of \eqref{spline11} if and only if $x_0 \in \esp(h_0).$ In fact, if $x_0$ is a solution of \eqref{spline11} then $Vx_0=Vh_0=f_0$ and $\K{Tx_0}{Tx_0}=\min \K{Tx}{Tx}$ for every $x \in \HH$ such that $Vx=f_0.$ Then $x_0=h_0+(x_0-h_0)\in h_0 + \ker V$ and, if $x \in \HH$ is such that $Vx=Vh_0=f_0$ then $x=h_0+z$ for some $z \in \ker V.$ Hence $x_0 \in \esp(h_0).$ 

Conversely, if $x_0 \in \esp(h_0)$ then $x_0=h_0+u$ for some $u \in \ker V$ so that $Vx_0=Vh_0=f_0$ and if $x:=h_0+z$ for some $z \in \ker V$ it holds that $Vx=Vh_0=f_0.$ Hence $x_0$ is a solution of \eqref{spline11}.
\end{obs}


%

The following result was proved in \cite[Lemma 3.4]{GiribetKreinsplines}.
\begin{lema}\label{Giribet} Let $f_0 \in \ran V$  and suppose that $Vx_0=f_0.$ Then $x_0$ is a solution of \eqref{spline11} if and only if $T(\ker V)$ is a nonnegative subspace of $\HH$ and $Tx_0\in [T(\ker V)]^{\perpi}.$
\end{lema}

In order to obtain solutions of \eqref{spline} that depend continuously on $h$ we give the following definition.

\begin{Def} Let $T \in L(\HH)$ and $V \in L(\HH).$ An operator $G \in L(\HH)$ is a \emph{bounded global solution} of \eqref{spline}  if	
\begin{equation*} \label{splineglobal}
Gh \in \esp(h) \mbox{ for every } h \in \HH.
\end{equation*}
\end{Def}
\smallskip

We are also interested in comparing the bounded global solution of \eqref{spline} to the \emph{operator spline problem}: 
let $(\HH, \K{ \ }{ \ })$ be a Krein space with signature operator $J.$ Fix the associated Hilbert space $(\HH, \PI{ \ }{ \ }_J).$ Given $T \in S_1$ and $V, B_0 \in L(\HH)$ such that $\ran V \subseteq \ran B_0,$ analyze if there exists 
\begin{equation}
	\label{uno}
\underset{X \in L(\HH), \ \ VX=B_0}{\min} \tr_{J}(X^{\#}T^{\#}TX).
\end{equation}

\vspace{0,3cm}
We begin by studying problem \eqref{uno}. 

\begin{prop} \label{teo1} Let $T \in S_1,$ $V, B_0 \in L(\HH)$ such that $\ran V \subseteq \ran B_0$ and $T(\ker V)$ is a nonnegative subspace of $\HH.$ Then the following are equivalent:
	\begin{itemize}
		\item [i)] there exists $\underset{VX=B_0}{\min} \tr_{J}(X^{\#}T^{\#}TX)$ for any signature operator $J;$
		\item [ii)] $\ran(V^{\dagger_J} B_0) \subseteq \ker V + \left[ T^{\#}T(\ker V)\right]^{\perpi}$ for any signature operator $J;$
		\item [iii)] the normal equation 	\begin{equation} \label{Normal2}
			(P_{\ker V}^J)^{\#} T^{\#}T (P_{\ker V}^JX+V^{\dagger_J} B_0)=0
		\end{equation} admits a solution for any signature operator $J;$
	\end{itemize}	
	In this case, for any signature operator $J,$ $$\underset{VX=B_0}{\min} \tr_{J}(X^{\#}T^{\#}TX)=\tr_{J}(X_0^{\#}T^{\#}TX_0),$$ where
	$X_0$ is any solution of equation \eqref{Normal2}.
\end{prop}

\begin{proof} Let $J$ be a signature operator of $\HH$ and fix the associated Hilbert space $(\HH, \PI{ \ }{ \ }_J).$
	
$i) \Leftrightarrow ii)$: If $VX=B_0$ then, by Douglas' Lemma \cite{Douglas}, $V^{\dagger_J}VX=P_{\ker V^{\perp}}^JX=V^{\dagger_J}B_0 \in L(\HH).$ Then $X=P_{\ker V}^JX +V^{\dagger_J}B_0$ and
\begin{equation*}
\min_{VX=B_0} \! \tr_{J}(X^{\#}T^{\#}TX) \! =\!\! \min_{X \in L(\HH)} \!
\tr_{J} [(P_{\ker V}^JX + V^{\dagger_J} B_0)^{\#}(T^{\#}T) (P_{\ker V^J}X + V^{\dagger_J} B_0)].
\end{equation*}
Then, by \cite[Theorem 5.8]{Contino33} and  \cite[Corollary 4.5]{Contino33}, problem \eqref{uno} admits a solution if only if $$\ran (V^{\dagger_J} B_0) \subseteq \ker V + \left[ T^{\#}T(\ker V)\right]^{\perpi}.$$
	
$ii) \Leftrightarrow iii)$: It follows from  \cite[Corollary 4.5]{Contino33}.
	
In this case, by \cite[Theorem 5.8]{Contino33} and \cite[Corollary 4.5]{Contino33}, for any signature operator $J,$ $$\underset{VX=B_0}{\min} \tr_{J}(X^{\#}T^{\#}TX)=\tr_{J}(X_0^{\#}T^{\#}TX_0),$$ where
$X_0$ is any solution of equation \eqref{Normal2}.
\end{proof}

\begin{prop} \label{prop11} Let $T \in S_1$ and $V, B_0 \in L(\HH)$ such that $\ran V \subseteq \ran B_0.$ 
Then there exists $X_0 \in L(\HH)$ such that, for any signature operator $J,$ $$X_0x \in \esp(V^{\dagger_J}B_0x) \mbox{ for every } x \in \HH$$ if and only if $X_0$ is a solution of  \eqref{uno} for any signature operator $J$ and $T(\ker V)$ is a nonnegative subspace of $\HH.$ 
\end{prop}

\begin{proof} Let $J$ be a signature operator of $\HH$ and fix the associated Hilbert space $(\HH, \PI{ \ }{ \ }_J).$ Suppose that $X_0 \in L(\HH)$ is a solution of \eqref{uno} and $T(\ker V)$ is a nonnegative subspace of $\HH.$  Since $VX_0=B_0,$ there exists $Y_0 \in L(\HH)$ such that $X_0=P_{\ker V}^JY_0+V^{\dagger_J}B_0$ and, by Proposition \ref{teo1}, $X_0$ is a solution of \eqref{Normal2}. But $P_{\ker V}^JY_0=P_{\ker V}^JX_0,$ so $Y_0$ is also a solution of \eqref{Normal2}. Then, by \cite[Corollary 4.5]{Contino33},
$$(P_{\ker V}^JY_0+V^{\dagger_J} B_0)^{\#}T^{\#}T(P_{\ker V}^JY_0+V^{\dagger_J} B_0) \! \leq \! (P_{\ker V}^JY+V^{\dagger_J} B_0)^{\#} T^{\#}T (P_{\ker V}^JY+V^{\dagger_J} B_0),\!$$ for every $Y \in L(\HH).$ Or equivalently, 
\begin{align*}
&\K{T(P_{\ker V}^JY_0+V^{\dagger_J} B_0) x}{T(P_{\ker V}^JY_0+V^{\dagger_J} B_0) x}\\
&\leq \K{T(P_{\ker V}^JY+V^{\dagger_J} B_0) x}{T(P_{\ker V}^JY+V^{\dagger_J} B_0) x},
\end{align*} for every $x \in \HH$ and $Y \in L(\HH).$
Let $ z \in \HH$ be arbitrary. For every $x \in \HH \setminus \{ 0\},$ there exists $Y\in L(\HH)$ such that $z=Yx.$ Therefore
\begin{align*}
&\K{T(P_{\ker V}^JY_0x+V^{\dagger_J} B_0x)}{T(P_{\ker V}^JY_0x+V^{\dagger_J} B_0x)}\\
&\leq \K{T(P_{\ker V}^Jz+V^{\dagger_J} B_0x)}{T(P_{\ker V}^Jz+V^{\dagger_J} B_0x)},
\end{align*} for every $x, z \in \HH.$ 
Then $X_0x \in  V^{\dagger_J} B_0x+ \ker V,$  
$$\K{TX_0x}{TX_0x}  \leq \K{Th}{Th}, \mbox{ for every } h  \in V^{\dagger_J} B_0x+ \ker V$$ and $X_0x \in \esp(V^{\dagger_J}B_0x).$
	
Conversely, suppose that, $X_0x \in \esp(V^{\dagger_J}B_0x) \mbox{ for every } x \in \HH$ and  for any signature operator $J.$ Then, by Lemma \ref{Giribet}, $T(\ker V)$ is a nonnegative subspace of $\HH.$ Also, for every $x \in \HH$ and any signature operator $J,$ $X_0x \in V^{\dagger_J} B_0x+ \ker V$ and 
$$\K{TX_0x}{TX_0x}  \leq \K{Th}{Th}, \mbox{ for every } h  \in V^{\dagger_J} B_0x+ \ker V.$$ It follows that $VX_0=B_0,$ and for any signature operator $J,$ 
	$$\K{TX_0x}{TX_0x}  \leq \K{T(P_{\ker V}^Jz+V^{\dagger_J} B_0x)}{T(P_{\ker V}^Jz+V^{\dagger_J} B_0x)},$$ for every $x, z \in \HH.$ In particular, given $Y \in L(\HH),$ consider $z=Yx.$ Then, for any signature operator $J,$  
	$$\K{TX_0x}{TX_0x}  \leq \K{T(P_{\ker V}^JY+V^{\dagger_J} B_0)x}{T(P_{\ker V}^JY+V^{\dagger_J} B_0)x},$$ for every $x\in \HH$ and $Y \in L(\HH).$ Fix $J$ and let $\{e_n : n \in \mathbb{N}\}$ be any orthonormal basis in $(\HH, \PI{ \ }{ \ }_J).$ Then 
	$$\K{TX_0e_n}{TX_0e_n}  \leq \K{T(P_{\ker V}^JY+V^{\dagger_J} B_0)e_n}{T(P_{\ker V}^JY+V^{\dagger_J} B_0)e_n},$$ for every $n\in \mathbb{N}$ and $Y \in L(\HH).$ Hence
	$$\tr_{J}(X_0T^{\#}TX_0) \leq \tr_{J}((P_{\ker V}^JY+V^{\dagger_J} B_0)^{\#}T^{\#}T(P_{\ker V}^JY+V^{\dagger_J} B_0)),$$ for every $Y \in L(\HH).$
	Therefore, $X_0$ is a solution of \eqref{uno}.
\end{proof}

Next we give necessary and sufficient conditions for the operator spline problem \eqref{uno} to have a solution for every $B_0 \in L(\HH).$ This result also shows that the condition that guarantees the existence of a bounded global solution of the indefinite spline problem \eqref{spline} is equivalent to the operator spline problem \eqref{uno} to have a solution for every $B_0 \in L(\HH).$

\begin{thm} Let $T \in S_1$ and $V \in L(\HH).$ Then the following are equivalent:
	\begin{enumerate}
		\item [i)] there exists $\underset{VX=B_0}{\min}\tr_{J}(X^{\#}T^{\#}TX)$ for any signature operator $J,$ for every $B_0 \in L(\HH)$ such that $\ran B_0 \subseteq \ran V$ and $T(\ker V)$ is a nonnegative subspace of $\HH;$   
		\item [ii)] there exists a bounded global solution of \eqref{spline};
		\item [iii)] $T^{\#}T$ is $\ker V$ complementable and $T(\ker V)$ is a nonnegative subspace of $\HH;$   
		\item [iv)]  $\esp(h_0)$ is nonempty for every $h_0\in \HH.$ 
	\end{enumerate} 
In this case, for any signature operator $J,$ 
	$$\underset{VX=B_0}{\min}\tr_{J}(X^{\#}T^{\#}TX)=\tr_{J}(X_0^{\#}T^{\#}TX_0)=\tr_{J}((V^{\dagger} B_0)^{\#}(T^{\#}T)_{/ [\ker V]}(V^{\dagger} B_0)),$$ where
	$X_0$ is any solution of equation \eqref{Normal2}.
\end{thm}

\begin{proof} $i) \Leftrightarrow ii):$ Let $J$ be a signature operator of $\HH$ and fix the associated Hilbert space $(\HH, \PI{ \ }{ \ }_J).$
Suppose that $T(\ker V)$ is a nonnegative subspace of $\HH$ and  $X_0 \in L(\HH)$ is a solution of \eqref{uno} for $B_0=V.$ Consider $G:=X_0P_{\ker V^{\perp}}^J \in L(\HH).$ Then, by Proposition \ref{prop11}, 
	$$
	Gh=X_0(P_{\ker V^{\perp}}^Jh) \in \esp(P_{\ker V^{\perp}}^Jh) \mbox { for every }  h \in \HH.
	$$
	Note that $\esp(P_{\ker V^{\perp}}^Jh)=\esp(h)$, because $P_{\ker V^{\perp}}^Jh+\ker V=h+\ker V.$
	Hence
	$$
	Gh \in \esp(h) \mbox { for every }  h \in \HH,
	$$
	so that $G$ is a bounded global solution of \eqref{spline}.
	
	Conversely, suppose that $G \in L(\HH)$ is a bounded global solution of \eqref{spline} and $B_0 \in L(\HH).$ Let $J$ be any signature operator and let $X_0:=GV^{\dagger_J}B_0 \in L(\HH).$ Then
	$$X_0x=G(V^{\dagger_J}B_0x) \in \esp(V^{\dagger_J}B_0x) \mbox{ for every } x \in \HH.$$ Therefore, by Proposition \ref{prop11}, $X_0$ is a solution of \eqref{uno} for any signature operator $J$ and $T(\ker V)$ is a nonnegative subspace of $\HH.$
	
	$i) \Leftrightarrow iii):$ Let $J$ be a signature operator of $\HH$ and fix the associated Hilbert space $(\HH, \PI{ \ }{ \ }_J).$ Suppose that \eqref{uno} has a solution for every $B_0 \in L(\HH)$ such $\ran B_0 \subseteq \ran V.$  Then, by Proposition \ref{teo1}, $$\ran(V^{\dagger_J} B_0) \subseteq \ker V + \left[ T^{\#}T(\ker V)\right]^{\perpi}$$ for every $B_0$ such that $\ran B_0 \subseteq \ran V.$  
	Consider $B_0=V,$ then 
	$\ker V^{\perp_J}= \ran (V^{\dagger} V) \subseteq \ker V + \left[ T^{\#}T(\ker V)\right]^{\perpi}.$ 
	So that
	$\HH=\ker V \oplus_J \ker V^{\perp_J} \subseteq \ker V + \left[ T^{\#}T(\ker V)\right]^{\perpi}$ and $T^{\#}T$ is $\ker V$ complementable.
	
	Conversely, suppose that $T^{\#}T$ is $\ker V$ complementable.  Then 
	$$\HH=  \ker V + \left[ T^{\#}T(\ker V)\right]^{\perpi}$$ and $\ran (V^{\dagger_J} B_0)  \subseteq \ker V + \left[ T^{\#}T(\ker V)\right]^{\perpi}$ for any signature operator $J$ and for every $B_0$ such that $\ran B_0 \subseteq \ran V.$ Then, by Proposition \ref{teo1},  \eqref{uno} has a solution for any signature operator $J$ and for every $B_0 \in L(\HH)$ such that $\ran B_0 \subseteq \ran V.$ 
	
	$iii) \Leftrightarrow iv):$ Suppose that $\esp(h_0)$ is nonempty for every $h_0\in \HH.$ Let $y \in \ran V,$ then $y=Vh_0 \in \HH$ for some $h_0,$ and there exists $u \in \HH$ such that $u \in \esp(h_0).$ 
	Then, by Lemma \ref{Giribet}, $T(\ker V)$ is a nonnegative subspace of $\HH$ and $Tu\in [T(\ker V)]^{\perpi}=(T^{\#})^{-1}(\ker V^{\perpi}).$ So $u \in  (T^{\#}T)^{-1}(\ker V^{\perpi})=\left[T^{\#}T(\ker V)\right]^{\perpi}.$ Hence $$y=Vh_0=Vu \in V(\left[ T^{\#}T(\ker V)\right]^{\perpi}).$$ So that $\ran V \subseteq V\left(\left[T^{\#}T(\ker V)\right]^{\perpi}\right) \subseteq \ran V,$ and
	$$\HH=V^{-1}(\ran V)=V^{-1}\left(V\left(\left[T^{\#}T(\ker V)\right]^{\perpi}\right)\right)=\ker V+ \left[T^{\#}T(\ker V)\right]^{\perpi}.$$ Therefore $T^{\#}T$ is $\ker V$ complementable. 
	
	Conversely, if $iii)$ holds then $\ran V=V(\HH)=V(\left[T^{\#}T(\ker V)\right]^{\perpi}).$ Let $h_0 \in \HH,$ then $Vh_0 \in V(\left[T^{\#}T(\ker V)\right]^{\perpi}).$ So that $h_0 \in \ker V+\left[T^{\#}T(\ker V)\right]^{\perpi}.$ That is, $h_0=x_0+z$ for some $x_0 \in \left[T^{\#}T(\ker V)\right]^{\perpi}$ and $z \in \ker V.$  But, since $Vx_0=Vh_0$ and $Tx_0 \in \left[T(\ker V)\right]^{\perpi},$ by Lemma \ref{Giribet}, $x_0$ is a solution of $\min \K{Tx}{Tx} \mbox{ subject to } Vx=Vh_0.$ Hence, by similar arguments as those found in Remark \ref{remarksp}, $x_0 \in \esp(h_0).$

	In this case, by Proposition \ref{teo1}, for any signature operator $J$ and every $B_0 \in L(\HH)$ such that $\ran B_0 \subseteq \ran V,$ 
	 $$\underset{VX=B_0}{\min} \tr_{J}(X^{\#}T^{\#}TX)=\tr_{J}(X_0^{\#}T^{\#}TX_0),$$  where
	$X_0$ is any solution of equation \eqref{Normal2}.
		
	By the proof of \cite[Corollary 4.6]{Contino33}, there exists $Y_0 \in L(\HH)$ such that $Y_0$ is a solution of equation \eqref{Normal2} and, for any signature operator $J,$
	$$(V^{\dagger_J} B_0)^{\#}(T^{\#}T)_{/ [\ker V]}(V^{\dagger_J} B_0)=(P_{\ker V}^JY_0+V^{\dagger_J}B_0)^{\#}T^{\#}T(P_{\ker V}^JY_0+V^{\dagger_J}B_0).$$
	Take $X_0:=P_{\ker V}^JY_0+V^{\dagger_J}B_0,$ then $X_0$ is also a solution of \eqref{Normal2}. Hence, for any signature operator $J,$
	$$\underset{VX=B_0}{\min}\tr_{J}(X^{\#}T^{\#}TX)=\tr_{J}(X_0^{\#}T^{\#}TX_0)=\tr_{J}((V^{\dagger_J} B_0)^{\#}(T^{\#}T)_{/ [\ker V]}(V^{\dagger_J} B_0)).$$
\end{proof}

\subsection*{Smoothing problems}

Let $T, V \in L(\HH),$ $\rho \in \mathbb{R}\setminus \{0\}$ and $h_0\in \HH.$  A problem that is naturally associated with \eqref{spline} is to find
\begin{equation} 
\label{CSP}
	\underset{x\in \HH}{\min}(\K{Tx}{Tx}+\rho \K{Vx-h_0}{Vx-h_0}).
\end{equation} 
We will refer to \eqref{CSP} as the \textit{indefinite smoothing problem} and its solutions are called \textit{indefinite smoothing splines}. 

From now on $\rho \in \mathbb{R} \setminus\{0\}.$ As before, we also study the problem of finding a \emph{bounded global solution} of problem \eqref{CSP}; i.e., we analyze if there exists an operator $G\in L(\HH)$ such that
\begin{equation*} \label{GlobalCSP}
\K{TGh}{TGh}+\rho \K{VGh-h}{VGh-h} \!=\!	\underset{x\in \HH}{\min}(\K{Tx}{Tx}+\rho \K{Vx-h}{Vx-h}), 
\end{equation*}
for every  $h \in \HH.$


We are interested in characterizing bounded global solutions of \eqref{CSP} and comparing them with the solutions of the  following \emph{indefinite operator smoothing} problem: let $(\HH, \K{ \ }{ \ })$ be a Krein space with signature operator $J$ and fix the associated Hilbert space $(\HH, \PI{ \ }{ \ }_J).$ Given $T, V, B_0 \in S_1$ analyze the existence of
\begin{equation} 
\label{dos}
	\underset{X \in L(\HH)}{\min} \ [\tr_{J}((TX)^{\#}TX)+ \rho \tr_{J}((VX-B_0)^{\#}(VX-B_0))].
\end{equation}

\smallskip

Define $K, B_0': \HH \ra \HH \times \HH,$ 
\begin{eqnarray}
	Kh=(Th,Vh) \mbox{ for } h \in \HH, \label{eqK} \\ 
	B_0'h=(0,B_0h) \mbox{ for } h \in \HH \label{eqB}.
\end{eqnarray}
We consider the indefinite inner product on $\HH \times \HH$ 
\begin{equation} \label{eqHH}
\K{(h_1,h_2)}{(f_1,f_2)}_\rho:=\K{h_1}{f_1}+\rho\K{h_2}{f_2},
\end{equation} $(h_1,h_2),(f_1,f_2) \in \HH \times \HH.$ It is easy to see that $(\HH\times \HH, \K{\cdot}{\cdot}_{\rho})$ is a Krein space.

It is straightforward to check that $K^{\#}: \HH \times \HH \ra \HH,$ is $$K^{\#}(h_1,h_2)=T^{\#}h_1+\rho V^{\#}h_2 \mbox{ for } (h_1,h_2) \in \HH \times \HH$$ and $B_0'^{\#}: \HH \times \HH \ra \HH,$ is $$B_0'^{\#}(h_1,h_2)=\rho B_0^{\#}h_2 \mbox{ for } (h_1,h_2) \in \HH \times \HH.$$

\begin{obs} \label{remark1} It holds that $\ran K$ is nonnegative if and only if $T^{\#}T+\rho V^{\#}V \in L(\HH)^+$.
\end{obs}

\begin{lema}  \label{teo2} Let $T, V, B_0 \in L(\HH).$ Set $K$ and $B_0'$ as in \eqref{eqK} and \eqref{eqB}. Then there exists $X_0\in L(\HH)$ such that \begin{equation} \label{lemaK}
(KX_0-B_0')^{\#}(KX_0-B_0')=\underset{X \in L(\HH)}{\min}(KX-B_0')^{\#}(KX-B_0'),
\end{equation} 
where the order is the one induced in $L(\HH)$ by the cone of positive operators, if and only if $T^{\#}T+\rho V^{\#}V \in L(\HH)^+$ and $X_0$ is a solution of the normal equation
	\begin{equation} \label{Normal1}
		(T^{\#}T+\rho V^{\#}V)X=\rho V^{\#}B_0.
	\end{equation}	
\end{lema}

\begin{proof} It can be checked that \eqref{lemaK} holds if and only if $X_0x$  is an indefinite least squares solution of the equation $Kz=B_0'x$ for every $x \in \HH,$  if and only if, by Proposition \ref{propWLSS}, $X_0x$ is a solution of $K^{\#}(Ky-B_0'x)=0$ for every $x \in \HH$  and $\ran K$ is nonnegative, or equivalently $X_0$ is a solution of \eqref{Normal1} and, by Remark \ref{remark1}, $T^{\#}T+\rho V^{\#}V \in L(\HH)^+.$ 
%
%
\end{proof}

\begin{prop} \label{teo3} Let $T, V, B_0 \in S_1$ such that $T^{\#}T+\rho V^{\#}V \in L(\HH)^+.$ Then the following are equivalent:
	\begin{enumerate}
		\item [i)] there exists $\underset{X \in L(\HH)}{\min} \ [\tr_{J}((TX)^{\#}TX)+ \rho \tr_{J}((VX-B_0)^{\#}(VX-B_0))]$ for any signature operator $J$;
		\item [ii)] the normal equation $(T^{\#}T+\rho V^{\#}V)X=\rho V^{\#}B_0$ admits a solution.
	\end{enumerate}
\end{prop}

\begin{proof}  Let $K,$ $B_0'$ be as in \eqref{eqK} and \eqref{eqB}, respectively. Let $J$ be a signature operator of $\HH$ and fix the associated Hilbert space $(\HH, \PI{ \ }{ \ }_J).$
If $X_0$ is a solution of the normal equation \eqref{Normal1}, then by Lemma \ref{teo2}, $$(KX_0-B_0')^{\#}(KX_0-B_0')\leq (KX-B_0')^{\#}(KX-B_0'), \mbox{ for every }  X\in L(\HH).$$ 	
Let $\{e_n : n \in \mathbb{N}\}$ be any orthonormal basis in $(\HH, \PI{ \ }{ \ }_J).$Then, for every $X \in L(\HH)$ and $n \in \mathbb{N},$ 
\begin{align*}
&\K{(KX-B_0')e_n}{(KX-B_0')e_n}_\rho\\
&=\K{(TXe_n,(V-B_0)Xe_n)}{(TXe_n,(V-B_0)Xe_n)}_\rho\\
&=\K{(TX)^{\#}TXe_n}{e_n}+\rho \K{(VX-B_0)^{\#}(VX-B_0)e_n}{e_n}.
\end{align*}
So that, for every $X \in L(\HH)$ and $n \in \mathbb{N},$
\begin{align*}
&\K{(TX_0)^{\#}TX_0e_n}{e_n}+\rho \K{(VX_0-B_0)^{\#}(VX_0-B_0)e_n}{e_n} \\
&\leq \K{(TX)^{\#}TXe_n}{e_n}+\rho \K{(VX-B_0)^{\#}(VX-B_0)e_n}{e_n}.
\end{align*}
Hence
\begin{align*}
&\tr_{J}((TX_0)^{\#}TX_0)+\rho \tr_{J}((VX_0-B_0)^{\#}(VX_0-B_0))\\
& \leq \tr_{J}((TX)^{\#}TX)+\rho \tr_{J}((VX-B_0)^{\#}(VX-B_0)).
\end{align*}
Thus, \eqref{dos} admits a solution.
	
To prove the converse, consider $F: L(\HH) \rightarrow \mathbb{R}^{+},$ 
$$
F(X)=\tr_{J}((TX)^{\#}TX)+\rho \tr_{J}((VX-B_0)^{\#}(VX-B_0)).$$

By \cite[Lemma 5.7]{Contino33}, $F$ has a Fr\'echet derivative, $F \in \mc{C}^1(L(\HH), \mathbb{R})$ and, for every $X, Y \in L(\HH),$ 
$$DF(X)(Y)=2\real\tr_{J}(Y^{\#}(T^{\#}TX+\rho V^{\#}(VX-B_0)).$$
Suppose that $X_0 \in L(\HH)$ is a global minimum  of $\tr_{J}((TX)^{\#}TX)+\rho \tr_{J}((VX-B_0)^{\#}(VX-B_0)).$ Then $X_0$ is a global minimum of $F$ and, since $F$ is a $\mc{C}^1$-function,
	$$DF(X_0)(Y)= 0, \mbox{ for every } Y \in L(\HH).$$
	So that, for every $Y \in L(\HH),$ 
	$$\real\tr_{J}(Y^{\#}(T^{\#}TX_0+\rho V^{\#}(VX_0-B_0))=0.$$

Let $Y_0:=JZ_0 \in L(\HH),$ where $Z_0:=T^{\#}TX_0+\rho V^{\#}(VX_0-B_0).$ Since $J^{\#}=J,$ $$\real \tr_J(Z_0^{\#}JZ_0)=\real \tr(JZ_0^{\#}JZ_0)=\real \tr(Z_0^{*_J}Z_0)=0.$$ Therefore $Z_0=0,$ or
$(T^{\#}T+\rho V^{\#}V)X_0=\rho V^{\#}B_0.$
\end{proof}

In  \cite{[Mit]}, Mitra defined the optimal inverses for matrices in order to  study the existence of solutions of inconsistent linear systems under seminorms defined by positive semidefinite matrices.  Mitra's concept was extended to Hilbert spaces in \cite{[CorFonMae16]}. Here we extend it to the Krein space setting:

\begin{Def}
	Given operators $A \in L(\HH)$ and $W\in L(\HH\times \HH)^s$ an \emph{indefinite $W$-optimal inverse} of $A$ is an operator $G \in L(\HH)$ such that 
	$$\K{W\begin{pmatrix}
			Gh \\ AGh-h
	\end{pmatrix}}{\begin{pmatrix}
	Gh \\ AGh-h
\end{pmatrix}}_\rho= \underset{x \in \HH}{\min} \ \K{W\begin{pmatrix}
x \\ Ax-h
\end{pmatrix}}{\begin{pmatrix}
x \\ Ax-h
\end{pmatrix}}_\rho,$$ for every $h \in \HH.$ 
\end{Def}

Consider  $W$ with the following block form in $\HH \times \HH,$
\begin{equation} \label{AblockformB}
	\left( \begin{array}{cc} 
		W_{11} & W_{12} \\
		W_{12}^{\#} & W_{22}\\
	\end{array}
	\right),
\end{equation}
where $W_{11}, W_{22} \in L(\HH)^s$ and $W_{12} \in L(\HH).$ 

\begin{prop}[{c.f.  \cite[Theorem 2.1]{[CorFonMae16]}, \cite[Theorem 4.2]{[Mit]}}] \label{MitraKrein} Let $A \in L(\HH)$ and $W\in L(\HH\times \HH)^s$ with matrix decomposition \eqref{AblockformB}. 
Then $A \in L(\HH)$ admits an indefinite $W$-optimal inverse if and only if $W_{11}+W_{12}A+\rho A^{\#}(W_{12}^{\#}+W_{22}A) \in L(\HH)^+$ and the equation
\begin{equation} \label{AoptinvB}
(W_{11}+W_{12}A+\rho A^{\#}W_{12}^{\#}+\rho A^{\#} W_{22}A)X=W_{12}+\rho A^{\#}W_{22}
\end{equation}
admits a solution. In this case, the set of indefinite $W$-optimal inverses of $A$ is the set of solutions of \eqref{AoptinvB}.
\end{prop}

\begin{proof} Suppose that $G \in L(\HH)$ is an indefinite $W$-optimal inverse of $A$ then, for every $h, x \in \HH,$
\begin{align*}
&\K{W\begin{pmatrix}
		Gh \\ AGh-h
\end{pmatrix}}{\begin{pmatrix}
		Gh \\ AGh-h
\end{pmatrix}}_\rho\\
&\leq \K{W\left[\begin{pmatrix} Gh \\ AGh-h
	\end{pmatrix}+ \begin{pmatrix}
	x-Gh \\ A(x-Gh)
\end{pmatrix}\right]}{\begin{pmatrix} Gh \\ AGh-h
\end{pmatrix}+ \begin{pmatrix}
x-Gh \\ A(x-Gh)
\end{pmatrix}}_\rho,
\end{align*} or equivalently, 
\begin{align*}
\!\!\!&\K{W\begin{pmatrix}
		Gh \\ AGh-h
\end{pmatrix}\!}{\!\begin{pmatrix}
		Gh \\ AGh-h
\end{pmatrix}}_\rho \! \!\!\leq\!\! \K{W\!\left[\begin{pmatrix} Gh\\ AGh-h
		\end{pmatrix}+t \begin{pmatrix} z \\ Az
	\end{pmatrix}\right]\!}{\!\begin{pmatrix} Gh\\ AGh-h
\end{pmatrix}+t \begin{pmatrix} z \\ Az
\end{pmatrix}}_\rho\\
&=\K{W\!\begin{pmatrix}
			Gh \\ AGh-h
	\end{pmatrix}}{\begin{pmatrix}
			Gh \\ AGh-h
	\end{pmatrix}}_\rho+t^2\K{W\begin{pmatrix}
	z \\ Az
\end{pmatrix}}{\begin{pmatrix}
z\\ Az
\end{pmatrix}}_\rho\\
&+ 2 t \real\K{W\begin{pmatrix}
Gh \\ AGh-h
\end{pmatrix}}{\begin{pmatrix}
z \\ Az
\end{pmatrix}}_\rho
\end{align*}
for every $h, z \in \HH$ and $t \in \mathbb{R}.$ Therefore
$$0 \leq t^2\K{W\begin{pmatrix}
		z \\ Az
\end{pmatrix}}{\begin{pmatrix}
		z\\ Az
\end{pmatrix}}_\rho+ 2 t \real\K{W\begin{pmatrix}
		Gh \\ AGh-h
\end{pmatrix}}{\begin{pmatrix}
		z \\ Az
\end{pmatrix}}_\rho,$$
for every $h, z \in \HH$ and $t \in \mathbb{R}.$ 

A standard argument shows that $\real\K{W\begin{pmatrix}
		Gh \\ AGh-h
\end{pmatrix}}{\begin{pmatrix}
		z \\ Az
\end{pmatrix}}_\rho=0,$ for every $h, z \in \HH.$ 

In a similar way if $s:=it,$ $t \in \mathbb{R},$ it follows that $\im \K{W\begin{pmatrix}
Gh \\ AGh-h
\end{pmatrix}}{\begin{pmatrix}
z \\ Az
\end{pmatrix}}_\rho=0,$ for every $h, z \in \HH.$ 
Then
\begin{align*}
	0&=\K{W\begin{pmatrix}
			Gh \\ AGh-h
	\end{pmatrix}}{\begin{pmatrix}
			z \\ Az
	\end{pmatrix}}_\rho=\K{\begin{pmatrix}
			W_{11}Gh+W_{12} (AGh-h)\\
			W_{12}^{\#}Gh+W_{22}(AGh-h)
		\end{pmatrix}
	}{\begin{pmatrix}
			z \\Az
	\end{pmatrix}}_\rho\\
	&=\K{W_{11}Gh+W_{12} (AGh-h)}{z}+\rho \K{W_{12}^{\#}Gh+W_{22}(AGh-h)}{Az}\\
	&=\K{W_{11}Gh+W_{12} (AGh-h)+\rho A^{\#}(W_{12}^{\#}Gh+W_{22}(AGh-h))}{z},
\end{align*}
for every $h, z \in \HH,$ and 
\begin{align*}
0 &\leq \K{W\begin{pmatrix}
		z \\ Az
\end{pmatrix}}{\begin{pmatrix}
		z\\ Az
\end{pmatrix}}_\rho=\K{\begin{pmatrix}
W_{11}z+W_{12}Az\\
W_{12}^{\#}z+W_{22}Az
\end{pmatrix}}{\begin{pmatrix}
z\\ Az
\end{pmatrix}}_\rho\\
&=\K{(W_{11}+W_{12}A)z}{z}+\rho\K{A^{\#}(W_{12}^{\#}+W_{22}A)z}{z},
\end{align*}
for every $z \in \HH.$ 
Hence $$W_{11}G+W_{12} (AG-I)+\rho A^{\#} (W_{12}^{\#}G+W_{22}(AG-I))=0,$$ i.e., $G$ is a solution of
\eqref{AoptinvB} and $W_{11}+W_{12}A+\rho A^{\#}(W_{12}^{\#}+W_{22}A) \in L(\HH)^+.$ 
 
 Conversely, if $W_{11}+W_{12}A+\rho A^{\#}(W_{12}^{\#}+W_{22}A) \in L(\HH)^+$ then $\K{W \begin{pmatrix}
z \\ Az
 \end{pmatrix}}{\begin{pmatrix}
 z \\ Az
\end{pmatrix}}_\rho\geq 0$ for every $z \in \HH.$ Also, if $G$ is a solution of \eqref{AoptinvB}, then
$\K{W\begin{pmatrix}
		Gh \\ AGh-h
\end{pmatrix}}{\begin{pmatrix}
		z \\ Az
\end{pmatrix}}_\rho=0,$ for every $h, z \in \HH.$ Therefore, for every $h, x \in \HH,$
\begin{align*}
&\K{W\left[\begin{pmatrix} Gh \\ AGh-h
	\end{pmatrix}+ \begin{pmatrix}
		x-Gh \\ A(x-Gh)
	\end{pmatrix}\right]}{\begin{pmatrix} Gh \\ AGh-h
	\end{pmatrix}+ \begin{pmatrix}
		x-Gh \\ A(x-Gh)
\end{pmatrix}}_\rho \\ &=\K{W\begin{pmatrix}
Gh \\ AGh-h
\end{pmatrix}}{\begin{pmatrix}
Gh \\ AGh-h
\end{pmatrix}}_\rho+\K{W \begin{pmatrix}
			x-Gh \\ A(x-Gh)
	\end{pmatrix}}{\begin{pmatrix}
		x-Gh \\ A(x-Gh)
	\end{pmatrix}}_\rho\\
&+ 2 \real\K{W\begin{pmatrix}
		Gh \\ AGh-h
\end{pmatrix}}{\begin{pmatrix}
	x-Gh \\ A(x-Gh)
\end{pmatrix}}_\rho\\
&\geq\K{W\begin{pmatrix}
		Gh \\ AGh-h
\end{pmatrix}}{\begin{pmatrix}
		Gh \\ AGh-h
\end{pmatrix}}_\rho.
\end{align*}
Then $G$ is an indefinite $W$-optimal inverse of $A.$

In this case, we have established that the set of indefinite $W$-optimal inverses of $A$ is the set of solutions of \eqref{AoptinvB}.
\end{proof}


\begin{thm} \label{cor31}  Let $T, V \in S_1.$ Then the following are equivalent:
	\begin{itemize}
		\item [i)] there exists $\underset{X \in L(\HH)}{\min} \ [\tr_{J}((TX)^{\#}TX)+ \rho \tr_{J}((VX-B_0)^{\#}(VX-B_0))]$ for any signature operator $J$ and for any $B_0 \in S_1,$ and $T^{\#}T+\rho V^{\#}V \in L(\HH)^+;$
		\item [ii)] $\ran V^{\#} \subseteq \ran (T^{\#}T+\rho V^{\#}V)$ and $T^{\#}T+\rho V^{\#}V \in L(\HH)^+;$
		\item [iii)] there exists $\underset{x\in \HH}{\min }(\K{Tx}{Tx}+\rho \K{Vx-h}{Vx-h}),$ for every $h \in \HH;$
		\item[iv)]  $V$ admits an indefinite $\begin{pmatrix} T^{\#}T & 0 \\ 0 & I
		\end{pmatrix}$-optimal inverse;
		\item [v)] there exists a bounded global solution of the indefinite smoothing problem \eqref{CSP}.
	\end{itemize}		
In this case, the $\begin{pmatrix} T^{\#}T & 0 \\ 0 & I
\end{pmatrix}$-optimal inverses of $V$ are the global solution of the indefinite smoothing problem \eqref{CSP}.
\end{thm}
\begin{proof} 
	$i) \Leftrightarrow ii)$: Suppose that \eqref{dos} has a minimum for every $B_0 \in S_1$ and $T^{\#}T+\rho V^{\#}V \in L(\HH)^+.$ Then, by Proposition \ref{teo3}, $\ran (V^{\#}B_0) \subseteq \ran (T^{\#}T+\rho V^{\#}V).$ 
	Let $h_0 \in \HH.$ Then, there exists $B_0 \in S_1$ such that $h_0=B_0x,$ for some $x\in \mathcal H$. Therefore
	$$
	V^{\#}h_0=V^{\#}B_0x \in \ran (V^{\#}B_0) \subseteq \ran (T^{\#}T+\rho V^{\#}V)
	$$
	and $ii)$ follows. 
	Conversely, suppose that  $ii)$ holds and consider  $B_0 \in S_1.$ Then  $\ran (V^{\#}B_0) \subseteq \ran (T^{\#}T+\rho V^{\#}V).$ Hence, by Douglas' Lemma there exists a solution of \eqref{Normal1}. Therefore, by Proposition \ref{teo3}, \eqref{dos} has a solution for every $B_0 \in S_1.$ 
	
	$ii) \Leftrightarrow iii)$: Let $K$ be as in \eqref{eqK}. Given $h_0\in \HH$, it holds that 
	$$\underset{x\in \HH}{\min }(\K{Tx}{Tx}+\rho \K{Vx-h_0}{Vx-h_0})=\underset{x \in \HH}{\min}  \K{Kx-(0,h_0)}{Kx-(0,h_0)}_\rho$$ 
	exists if and only if, by Proposition \ref{propWLSS}, $\ran K$ is nonnegative and the normal equation $K^{\#}K x=K^{\#}(0, h_0)$ has a solution for every $h_0 \in \HH$ ; equivalently $V^{\#}h_0=(T^{\#}T+\rho V^{\#}V)x$ has a solution for every $h_0 \in \HH$ and $T^{\#}T+\rho V^{\#}V \in L(\HH)^+.$ Therefore, $\underset{x\in \HH}{\min }(\K{Tx}{Tx}+\rho \K{Vx-h_0}{Vx-h_0})$ exists, for every $h_0\in \mathcal{H}$ if and only if $\ran V^{\#} \subseteq \ran (T^{\#}T+\rho V^{\#}V)$ and $T^{\#}T+\rho V^{\#}V \in L(\HH)^+.$
	
	$ii) \Leftrightarrow iv)$: By Proposition \ref{MitraKrein}, $iv)$ holds if and only if $T^{\#}T+\rho V^{\#}V \in L(\HH)^+$ and the equation $(T^{\#}T+\rho V^{\#}V)X=\rho V^{\#}$ admits a solution, if and only if $ii)$ holds, by Douglas' Lemma.
	
	$iv) \Leftrightarrow v)$: The operator $G \in L(\HH)$ is an indefinite $\begin{pmatrix} T^{\#}T & 0 \\ 0 & I
	\end{pmatrix}$-optimal inverse of $V$ if and only if for every $h \in \HH$,  
\begin{align*}
\K{\begin{pmatrix} T^{\#}T & 0 \\ 0 & I
	\end{pmatrix}\begin{pmatrix}
		Gh \\ VGh-h
\end{pmatrix}}{\begin{pmatrix}
		Gh \\ VGh-h
\end{pmatrix}}_\rho&=\K{TGh}{TGh}+\rho\K{VGh-h}{VGh-h}\\
&\leq \K{\begin{pmatrix} T^{\#}T & 0 \\ 0 & I
	\end{pmatrix}\begin{pmatrix}
		x \\ Vx-h
\end{pmatrix}}{\begin{pmatrix}
		x \\ Vx-h
\end{pmatrix}}_\rho\\
&= \K{Tx}{Tx}+\rho \K{Vx-h}{Vx-h},
\end{align*}
for every $x \in \mathcal H,$ that is, $G$ is a bounded global solution of \eqref{CSP}.

In this case, we have established that the $\begin{pmatrix} T^{\#}T & 0 \\ 0 & I
\end{pmatrix}$-optimal inverses of $V$ are the global solution of the indefinite smoothing problem \eqref{CSP}.
\end{proof}

\subsection*{Acknowledgments} 
M.~Contino was supported by CONICET PIP 11220200102127CO, by María Zambrano Postdoctoral Grant CT33/21 at Universidad Complutense de Madrid financed by the Ministry of Universities with Next Generation EU funds, and by Grant CEX2019-000904-S funded by MCIN/AEI/10.13039/501100011033. 



\begin{thebibliography}{10}

\bibitem{Shorted2} W.~N.~Anderson, and G.~E.~Trapp, {\it Shorted Operators II}, SIAM J. Appl. Math. \textbf{28} (1975), 60--71.


\bibitem{[AnsLau68]} P.~M.~Anselone, and P.~J.~Laurent, {\it A general method for the construction of interpolating or smoothing spline-functions}, Numer. Math., \textbf{12} (1968), 66--82.

\bibitem{AntCorSto06} J.~Antezana, G.~Corach, and D.~Stojanoff, {\it Bilateral shorted operators and parallel sums}, Linear Algebra Appl. \textbf{414} (2006), 570--588.

\bibitem{[Att65]}M.~Atteia, {\it Generalization de la d\'efinition et des propri\'et\'es des spline-fonctions}, C.R. Acad. Sci. Paris \textbf{260} (1965), 3550--3553.

\bibitem{Azizov} T.~Y.~Azizov, and I.~S.~Iokhvidov, {\it Linear operators in spaces with and indefinite metric}, John Wiley and Sons, New York, 1989.

\bibitem{Bognar} J.~Bogn\'ar, {\it Indefinite inner product spaces}, Springer, Berlin, 1974.


\bibitem{CanuII} S.~Canu, C.~S.~Ong, and X.~Mary, \emph{Splines with non positive kernels}, Proceedings of the 5th International ISAAC Congress, (2005), 1--10.

\bibitem{Canu} S.~Canu, C.~S.~Ong, X.~Mary, and A.~Smola, \emph{Learning with non-positive kernels}, Proc. of the 21st International Conference on Machine Learning (2004), 639--646.

\bibitem{[ChaLenMil96]} R.~Champion~R., C.~T.~Lenard, and T.~M.~Mills, {\it An introduction to abstract splines}, Math. Sci. \textbf{21} (1996), 8-16.

\bibitem{[ChaLenMil00]} R.~Champion~R., C.~T.~Lenard, and T.~M.~Mills, {\it A variational approach to splines}, The ANZIAM J. \textbf{42} (2000), 119-135.
	
\bibitem{Contino3g} M.~Contino, M.~E.~Di Iorio y Lucero, and G.~Fongi, {\it Global solutions of approximation problems in Hilbert spaces}, Linear and Multilinear Algebra \textbf{69} (2021), 2510--2526.

\bibitem{Contino} M.~Contino, J.I.~Giribet, and A.~Maestripieri, {\it Weighted Procrustes problems}, J. Math. Anal. Appl. \textbf{445} (2017), 443--458.


\bibitem{Contino3} M.~Contino, A.~Maestripieri, and S.~Marcantognini, {\it Operator least squares problems and Moore-Penrose inverse in Krein Spaces}, Integr. Equ. Oper. Theory \textbf{90} (2018), 32. 

\bibitem{Contino4} M.~Contino, A.~Maestripieri, and S.~Marcantognini, {\it Schur complements of selfadjoint Krein space operators}, Linear Algebra Appl. \textbf{581} (2019), 214--246.

\bibitem{Contino33} M.~Contino, A.~Maestripieri, and S.~Marcantognini, {\it Weighted operator least squares problems and the J-trace in Krein spaces}, Math. Nachrichten  \textbf{293} (2019), 1730--1745. 

\bibitem{[CorFonMae16]} G.~Corach, G.~Fongi, and A.~Maestripieri, {\it Optimal inverses and abstract splines}, Linear Algebra Appl. \textbf{496} (2016), 182--192.

\bibitem{[CorFonMae13]} G.~Corach, G.~Fongi, and A.~Maestripieri, {\it Weighted projections into closed subspaces}, Studia Math. \textbf{216} (2013), 131--148.


\bibitem{Spline} G.~Corach, A.~Maestripieri, and D.~Stojanoff, {\it Oblique projections and abstract splines}, J. Approx. Theory \textbf{117} (2002), 189--206.



\bibitem {Douglas} R.~G.~Douglas, {\it On majorization, factorization and range inclusion of operators in Hilbert space}, Proc. Amer. Math. Soc. \textbf{17} (1966), 413--416.


\bibitem{DR1} M.~A.~Dritschel, and J.~Rovnyak, {\it Operators on indefinite inner product spaces}, Lect. Oper. Theory Its Appl. \textbf{3} (1996), 141--232.


\bibitem{GiribetKrein} J.I.~Giribet, A.~Maestripieri, and F.~Mart\'inez Per\'ia, {\it A geometrical approach to indefinite least squares problems}, Acts Appl. Math. \textbf{111} (2010), 65--81.

\bibitem{GiribetKreinsplines} J.~I.~Giribet, A.~Maestripieri, and F.~Mart\'inez Per\'ia, {\it Abstract splines in Krein spaces}, J. Math. Anal. Appl.  \textbf{369} (2010), 423--436.


\bibitem{GiribetKreinII} J.I.~Giribet, A.~Maestripieri, and F.~Mart\'inez Per\'ia, {\it Indefinite least-squares problems and pseudo-regularity,},  J. Math. Anal. Appl. \textbf{430} (2016), 895--908.


\bibitem{JOTA} S.~Gonz\'alez Zerbo, A.~Maestripieri, and F.~Mart\'inez Per\'ia, {\it Indefinite Abstract Splines with a Quadratic Constraint}, Journal of Optimization Theory and Applications \textbf{186} (2020), 209--225.

\bibitem{UMA} S.~Gonz\'alez Zerbo, A.~Maestripieri, and F.~Mart\'inez Per\'ia, \emph{Regularization of an indefinite abstract interpolation problem with a quadratic constraint}, Rev. de la Union Mat. Argentina \textbf{62} (2020), 475--490.



\bibitem{Hassibietal2} B.~Hassibi, A.~H.~Sayed, and T.~Kailath, \emph{$H^{\infty}$ optimality criteria for LMS and backpropagation}, Adv. Neural
Inf. Process. Syst. \textbf{6} (1994), 351--359.


\bibitem{Hassibietal3} B.~Hassibi, A.~H.~Sayed, and T.~Kailath, \emph{Indefinite-Quadratic Estimation and Control. A Unified Approach to $\HH^2$ and $\HH^{\infty}$ Theories}, Studies in Applied and Numerical Mathematics, SIAM, Philadelphia (1999).

\bibitem{Hassibietal} B.~Hassibi, A.~H.~Sayed, and T.~Kailath, {\it Inertia conditions for the minimization of quadratic forms in indefinite metric spaces}, Oper. Theory Adv. Appl. \textbf{87} (1996), 309--347.


\bibitem{HassibipartI} B.~Hassibi, A.~H.~Sayed, and T.~Kailath, {\it Linear estimation in Krein spaces - part I: theory}, IEEE Trans. Autom. Control \textbf{41} (1996), 18--33.

\bibitem{HassibipartII} B.~Hassibi, A.~H.~Sayed, and T.~Kailath, {\it Linear estimation in Krein spaces - part II: applications}, IEEE Trans. Autom. Control \textbf{41} (1996), 33--49.

\bibitem{Holladay} J.~C.~Holladay, {\it A smoothest curve approximation}, Math. Tables Aids to Comp. \textbf{11} (1957), 233--243.


\bibitem{Ulric} U.~Kintzel, {\it Procrustes problems in finite dimensional indefinite scalar product spaces}, Linear Algebra Appl. \textbf{402} (2005), 1--28.

\bibitem {Krein} M.~G.~Krein,  {\it The theory of self-adjoint extensions of semibounded Hermitian operators and its applications}, Mat. Sb. (N.S.) \textbf{20} (62) (1947), 431--495.


\bibitem{Loosli} G.~Loosli, S.~Canu, and C.~S.~Ong, \emph{Learning SVM in Krein spaces}, IEEE Trans. Pattern Anal. Machine Intelligence \textbf{38} (2016), 1204--1216.

\bibitem{[Mit]} S.~K.~Mitra, {\it Optimal inverse of a matrix}, Sankhya Ser. A. \textbf{37} (1975), 550--563.
	
\bibitem{Mitra} S.~K.~Mitra, and C.~R.~Rao, {\it Projections under seminorms and generalized Moore Penrose inverses and operator ranges}, Linear Algebra Appl. \textbf{9} (1974), 155--167.


\bibitem{Ringrose} J.R.~Ringrose, {\it Compact non-self-adjoint operators}, Van Nostrand Reinhold Co., New York, 1971.

%



\end{thebibliography}

\end{document}